\documentclass[11pt]{article}
\usepackage{latexsym,amsmath}
\def\fullpage {
\addtolength{\topmargin}{-2.2 cm}
\addtolength{\oddsidemargin}{-1.6 cm} \addtolength{\textwidth}{+2.9 cm}
\addtolength{\textheight}{+3.7 cm}}
\fullpage
\usepackage{amsthm,color}
\usepackage{amssymb}
\usepackage{epsfig}
\usepackage{graphicx}
\usepackage{amssymb}
\usepackage{multicol}
\usepackage{graphicx}
\usepackage{amsmath, epsfig}
\usepackage{amsthm}
\newtheorem{theorem}{Theorem}
\newtheorem{conj}{Conjecture}

\newtheorem{lemma}{Lemma}

\renewcommand{\phi}{\varphi}
\newtheorem{definition}{Definition}
\begin{document}
 \date{}

\title{\vspace{-0.7cm}Cycles in triangle-free graphs of  large chromatic number\thanks{The authors
thank Institute Mittag-Leffler (Djursholm, Sweden) for the hospitality and creative environment.}}

\author{
Alexandr Kostochka \thanks{Department of Mathematics, University of Illinois at Urbana--Champaign, IL, USA and
Sobolev Institute of Mathematics, Novosibirsk 630090, Russia, kostochk@math.uiuc.edu, Research supported in part by NSF grant
 DMS-1266016 and by Grant NSh.1939.2014.1 of the President of
Russia for Leading Scientific Schools.}
\and Benny Sudakov\thanks{Department of Mathematics, ETH, 8092 Zurich.
Email: benjamin.sudakov@math.ethz.ch.
Research supported in part by SNSF grant 200021-149111 and
by a USA-Israel BSF grant.
}
\and Jacques Verstra\"ete\thanks{
Department of Mathematics, University of California, San Diego, CA, USA, jverstra@math.ucsd.edu.
Research supported in part by an Alfred P. Sloan Research Fellowship and NSF Grant DMS-0800704.}
}

\maketitle

\vspace{-0.3in}

\begin{abstract}
More than twenty years ago Erd\H{o}s conjectured~\cite{E1} that a triangle-free graph $G$ of chromatic number
$k \geq k_0(\varepsilon)$ contains cycles of at least $k^{2 - \varepsilon}$
different lengths as $k \rightarrow \infty$.
{In this paper, we prove the
{stronger fact}} that every triangle-free graph $G$ of chromatic number $k \geq k_0(\varepsilon)$ contains cycles of $(\frac{1}{64} - \varepsilon)k^2 \log k$
consecutive lengths, and a cycle of length at least $(\tfrac{1}{4} - \varepsilon)k^2 \log k$.
As there exist triangle-free graphs of chromatic number $k$ with at most roughly $4k^2 \log k$
 vertices for large $k$, theses results are tight up to a constant factor. We also give new lower bounds on the circumference and the number of
 different  cycle lengths  for $k$-chromatic graphs in other monotone classes, in particular, for $K_r$-free graphs and graphs without
 odd cycles $C_{2s+1}$.
\end{abstract}

\section{Introduction}

It is well-known that every $k$-chromatic graph has a cycle of length at least $k$ for $k \geq 3$. In 1991, Gyarf\' as~\cite{G} proved a stronger statement,
namely, the conjecture by Bollob\' as and Erd\H os that
every graph of chromatic number $k \geq 3$ contains cycles of at least $\lfloor \frac{1}{2}(k - 1) \rfloor$ odd lengths.
This is best possible in view of any graph whose blocks are complete graphs of order $k$. Mihok and Schiermeyer~\cite{MS} proved
a similar result for even cycles: every graph  $G$  of chromatic number $k \geq 3$ contains cycles of at least $\lfloor \frac{k}{2}\rfloor - 1$ even lengths.
A consequence of the main result in~\cite{V} is that a graph of chromatic number $k \geq 3$ contains cycles of $\lfloor \frac{1}{2}(k - 1) \rfloor$ consecutive lengths. Erd\H{o}s~\cite{E1} made the following conjecture:

\begin{conj} \label{EC}
For every $\varepsilon > 0$, there exists $k_0(\varepsilon)$ such that for $k \geq k_0(\varepsilon)$, every triangle-free $k$-chromatic graph contains more than $k^{2-\epsilon}$ odd cycles of different lengths.
\end{conj}

The second and third authors proved~\cite{SV1} that if $G$ is a graph of average degree $k$ and girth at least five, then
$G$ contains cycles of $\Omega(k^2)$ consecutive even lengths, and in~\cite{SV2} it was shown that if an $n$-vertex graph of independence number at most $\frac{n}{k}$ is triangle-free, then it contains cycles of $\Omega(k^2\log k)$ consecutive lengths.

\subsection{Main Result}

 In this paper, we prove Conjecture \ref{EC} in the following stronger form:

\begin{theorem} \label{main}
For all $\varepsilon > 0$, there exists $k_0(\varepsilon)$ such that for $k \geq k_0(\varepsilon)$, every triangle-free $k$-chromatic graph $G$ contains a cycle
of length at least $(\tfrac{1}{4} - \varepsilon)k^2\log k$ as well as cycles of at least $(\frac{1}{64} - \varepsilon)k^2 \log k$ consecutive lengths.
\end{theorem}


Kim~\cite{K} was the first to construct a triangle-free graph with chromatic number $k$ and $\Theta(k^2 \log k)$ vertices.
Bohman and Keevash~\cite{BK} and Fiz Pontiveros, Griffiths and Morris~\cite{FGM} independently constructed a $k$-chromatic triangle-free graph with at most $(4 + o(1))k^2 \log k$ vertices as $k \rightarrow \infty$, refining the earlier construction of Kim~\cite{K}.
These constructions show that the bound in Theorem \ref{main} is tight up to a constant factor.

\subsection{Monotone Properties}

Theorem \ref{main} is a special case of a more general theorem on monotone properties. A graph property is called
{\em monotone} if it holds for all subgraphs of a graph which has this property, i.e., is preserved under deletion of edges and vertices.
Throughout this section, let $n_{\mathcal{P}}(k)$ denote the smallest possible order of a $k$-chromatic graph in a monotone property $\mathcal{P}$.

\begin{definition}
Let $\alpha \geq 1$ and let $f : [3,\infty) \rightarrow \mathbb R^+$. Then $f$ is $\alpha$-bounded
if $f$ is non-decreasing and whenever $y \geq x \geq 3$, $y^{\alpha} f(x) \geq x^{\alpha} f(y)$.
\end{definition}

For instance, any polynomial with positive coefficients is $\alpha$-bounded for some $\alpha \geq 1$. We stress that an $\alpha$-bounded
function is required to be a non-decreasing positive real-valued function with domain $[3,\infty)$.


\begin{theorem} \label{main1}
For all $\varepsilon > 0$ and $\alpha,m \geq 1$, there exists $k_1 = k_1(\varepsilon,\alpha,m)$ such that the following holds.
If $\mathcal{P}$ is a monotone property of graphs with $n_{\mathcal{P}}(k) \geq f(k)$ for $k \geq m$ and some $\alpha$-bounded function $f$,
then for $k \geq k_1$, every $k$-chromatic graph $G \in \mathcal{P}$ contains $(i)$ a cycle of length at least $(1 - \varepsilon)f(k)$ and $(ii)$ cycles of at least $(1 - \varepsilon)f(\frac{k}{4})$ consecutive lengths.
\end{theorem}

If $n_{\mathcal{P}}(k)$ itself is $\alpha$-bounded for some $\alpha$, then we obtain from Theorem \ref{main1} a tight result that a $k$-chromatic graph in $\mathcal{P}$ contains a cycle of length asymptotic to $n_{\mathcal{P}}(k)$ as $k \rightarrow \infty$. But proving that $n_{\mathcal{P}}(k)$ is $\alpha$-bounded for some $\alpha$ is probably difficult
for many properties, and
in the case $\mathcal{P}$ is the property of $F$-free graphs, perhaps is as difficult as obtaining asymptotic formulas for certain Ramsey numbers.
Even in the case of the property of triangle-free graphs, we have seen $n_{\mathcal{P}}(k)$ is known only up to a constant factor. We remark
that in Theorems \ref{main} and \ref{main1}, we have not attempted to optimize the quantities $k_0(\varepsilon)$ and $k_1(\varepsilon,\alpha,m)$.

\subsection{An application: $K_r$-free graphs}

As an example of an application of Theorem \ref{main1}, we consider the property $\mathcal{P}$ of $K_r$-free graphs. A lower bound for the quantity $n_{\mathcal{P}}(k)$ can be obtained by combining upper bounds for Ramsey numbers together with a lemma on colorings obtained
by removing maximum independent sets -- see Section \ref{main2proof}. In particular, we shall obtain
the following from Theorem \ref{main1}:

\begin{theorem} \label{main2}
If $G$ is a $k$-chromatic $K_{r+1}$-free graph, where $r,k \geq 3$, then $G$ contains
a cycle of length $\Omega(k^{\frac{r}{r - 1}})$, and cycles of
$\Omega(k^{\frac{r}{r - 1}})$ consecutive lengths as $k \rightarrow \infty$.
\end{theorem}

Theorem \ref{main2} is derived from upper bounds on the Ramsey numbers $r(K_r,K_t)$ combined with Theorem \ref{main1}. In general,
if for a graph $F$ one has $r(F,K_t) = O(t^{a} (\log t)^{-b})$ for some $a > 1$ and $b > 0$,
then any $k$-chromatic $F$-free graph has cycles of
\[ \Omega\bigl(k^{\frac{a}{a - 1}}(\log k)^{\frac{b}{a - 1}}\bigr)\]
consecutive lengths. We omit the technical details, since the ideas of the proof are identical to those used for Theorem \ref{main2}. These technical details are presented in the proof of Theorem \ref{main} in Section \ref{mainproof}, where $F$ is a triangle (in which case $a = 2$ and $b = 1$), and the same ideas can be used to slightly improve Theorem \ref{main2} by logarithmic factors using better bounds on $r(K_s,K_t)$ from results of
Ajtai, Koml\'os and Szemer\'edi \cite{AKS}. Similarly, if $C_\ell$ denotes the cycle of length
$\ell$, then it is known that $r(C_{2s + 1},K_t) = O(t^{1 + 1/s}(\log t)^{-1/s})$ -- see~\cite{Sudakov}. This in turn provides
cycles of $\Omega(k^{s+1} \log k)$ consecutive lengths in any $C_{2s+1}$-free $k$-chromatic graph, extending Theorem \ref{main}.

\bigskip

{\bf Notation and terminology.} For a graph $G$, let $c(G)$ denote the length of a longest cycle in $G$ and $\chi(G)$ the
chromatic number of $G$.
If $F \subset G$ and $S \subset V(G)$, let $G[F]$ and $G[S]$ respectively denote the subgraphs of $G$ induced by $V(F)$ and $S$.
A chord of a cycle $C$ in a graph is an edge of the graph joining two non-adjacent vertices on the cycle.
All logarithms in this paper are with the natural base.

\bigskip

{\bf Organization.} In the next section, we present the lemmas which will be used to prove Theorem \ref{main1}. Then in Section \ref{mainproof},
we apply Theorem \ref{main1} to obtain the proof of Theorem \ref{main}. Theorem \ref{main2} is proved in Section \ref{main2proof}.

\section{Lemmas}

\subsection{Vertex cuts in $k$-critical graphs}

When a small vertex cut is removed from a $k$-critical graph,
all the resulting components still have relatively high chromatic number:

\begin{lemma}\label{le0}
Let $G$ be a $k$-critical graph and let $S$ be a vertex cut of $G$. Then for any component $H$ of $G - S$,
$\chi(H) \geq k - |S|$.
\end{lemma}

\begin{proof}
If $|S|+\chi(H)\leq k-1$, then a $(k-1)$-coloring of $G-H$ (existing by the criticality of $G$) can be extended to
a $(k-1)$-coloring of $G$.
\end{proof}

\subsection{Nearly 3-connected subgraphs}

Our second lemma finds an almost 3-connected subgraph with high chromatic number in a graph with
high chromatic number.

\smallskip

\begin{lemma}\label{le1}
Let $k \geq 4$. For every 
  $k$-chromatic graph $G$, there is a graph $G^*$ and
an edge $e^*\in E(G^*)$ such that
\vspace{-0.1in}
\begin{center}
\begin{tabular}{lp{5in}}
{\rm (a)} & $G^*-e^* \subset G$ and $\chi(G^*-e^*) \geq k - 1$.\\
{\rm (b)} & $G^*$ is $3$-connected. \\
{\rm (c)} & $c(G^*)\leq c(G)$.\\
\end{tabular}
\end{center}
\end{lemma}

\begin{proof} Let $G'$ be a $k$-critical subgraph of $G$. Then $G'$ is 2-connected.
If $G'$ is $3$-connected, then the lemma holds for $G^*=G'$ with any $e \in E(G')$ as $e^*$. So suppose $G'$ is not $3$-connected.
Among all separating sets $S$ in $G'$ of size $2$ and components $F$ of $G'-S$, choose a pair $(S,F)$ with the minimum $|V(F)|$. If $S = \{u,v\}$, then
we let $G^*$ be induced by $V(F) \cup S$ plus the edge $e^* = uv$. 
   We claim $\chi(G^*-e^*) \geq k - 1$.
Since $G'$ is $k$-critical, there is a $(k - 1)$-coloring $\phi : V(G') \backslash V(F) \rightarrow \{1,2,\dots,k-1\}$ of $G' - V(F)$ and we
may assume $\phi(u) =k- 1$ and by renaming colors $\phi(v) \in \{1,k-1\}$. Suppose for a contradiction that there is a coloring
$\phi^* : V(G^*-e^*) \rightarrow \{1,2,\dots,k-2\}$ of $G^*-e^*$. If $\phi(v)=k-1$, then we let $\phi'(x)=\phi(x)$ if $x\in V(G')-V(F)$ and
$\phi'(x)=\phi^*(x)$ if $x\in V(F)$, and this $\phi'$ is a proper $(k-1)$-coloring of $G'$, a contradiction. Otherwise $\phi(v)=1$.
Then  we change
the names of colors in $\phi^*$ so that $\phi^*(v)=1$ and again let $\phi'(x)=\phi(x)$ if $x\in V(G')-V(F)$ and
$\phi'(x)=\phi^*(x)$ if $x\in V(F)$. Again
we have a proper $(k - 1)$-coloring of $G'$. This  contradiction proves  (a).

To prove (b), if $G^*$ has a separating set $S'$ with $|S'|=2$, then, since $uv \in E(G^*)$, it is also a separating set in $G$
and at least one component of $G' - S'$ is strictly contained in $F$. This contradicts the choice of $F$ and $S$.

For (c), let $C$ be a cycle in $G^*$ with $|C|=c(G^*)$. If $e^*\notin E(C)$, then $C$ is also a cycle in $G$, and thus
$c(G) \geq |C| = c(G^*)$. If $e^*\in E(C)$ and $G^* \neq G'$, then we obtain a longer cycle $C$ in $G'$ by replacing $e^*$
with a $uv$-path in $G' - V(F)$ -- note such a path exists since $G'$ is $2$-connected. This proves (c).
\end{proof}

\subsection{Finding cycles of consecutive lengths}

In this subsection we show how to go from longest cycles in graphs to cycles of many consecutive lengths.
We will need the following result from~\cite{V},
which is also implicit in the paper of Bondy and Simonovits~\cite{BS}:

\begin{lemma}[Lemma 2 in~\cite{V}] \label{jacques}
Let $H$
be a graph comprising a cycle with a chord. Let
$(A,B))$  be a nontrivial partition of $V(H)$.
Then $H$ contains $A,B$-paths of every positive length less than
$|H|$, unless $H$ is
bipartite with bipartition $(A,B)$.
\end{lemma}

\begin{lemma}\label{odd}
Let $k \geq 4$ and $\mathcal{Q}$ be a monotone class of graphs. Let $h(k,\mathcal{Q})$ denote the smallest possible length
of a longest cycle in any $k$-chromatic  graph in $\mathcal{Q}$.  Then every $4k$-chromatic
 graph in $\mathcal{Q}$ contains cycles of at least $h(k,\mathcal{Q})$ consecutive lengths.
\end{lemma}

\begin{proof}
Let $F$ be a connected subgraph of $G\in \mathcal{Q}$ with chromatic number at least $4k$ and let $T$ be a breadth-first search tree in $F$. Let $L_i$ be the set of vertices at distance exactly
$i$ from the root of $T$ in $F$. Then for some $i$, $H = F[L_i]$ has chromatic number at least $2k$.
Let $U$ be a breadth-first search tree in a component of $H$ with chromatic number at least $2k$ and let $M_i$ be the set of vertices at distance exactly
$i$ from the root of $U$ in $H$. Then for some $i$, $J = H[M_i]$ has chromatic number at least $k$.
Let $J'$ be a $k$-critical subgraph of $J$. Let $P$ be a longest path in $J'$, so that $|P| \geq h(k,\mathcal{Q})$. Since $J'$ has minimum
degree at least $k - 1 \geq 3$,  each of the ends of $P$  has at least two neighbors on $P$. In particular, there is a
path $P' \subset P$ of odd length with at least one chord, obtained by deleting at most one end of $P$, and $|P'| \geq h(k,\mathcal{Q}) - 1$.
Then the ends of $P'$ are joined by an even length path $Q \subset U$ that is internally disjoint from $P'$, and $C = Q \cup P'$
is a cycle of odd length plus a chord, with $|C| \geq h(k,\mathcal{Q}) + 1$. Let $\ell:=|C|$ and $H' = G[C]$.
Now $V(H') \subset L_i$ by construction. Let $T'$ be a
minimal subtree of $T$ whose set of leaves is $V(H')$. Then $T'$ branches at its root. Let $A$ be the set of leaves in some branch of $T'$,
and let $B = V(H') \backslash A$. Then $(A,B)$ is not a bipartition of $H'$, since $C$ has odd length, and therefore
by Lemma 1 in~\cite{V}, there exist paths $P_1,P_2,\dots,P_{\ell - 1} \subset H'$ such that $P_i$ has length $i$ and
one end of $P_i$ is in $A$ and one end of $P_i$ is in $B$, for $i = 1,2,\dots,\ell - 1$. Now for each path $P_i$, the ends of $P_i$
are joined by a path $Q_i$ of length $2r$, where $r$ is the height of $T'$ and $Q_i$ and $P_i$ are internally disjoint.
Therefore $P_i \cup Q_i$ is a cycle of length $2r + i$ for $i = 1,2,\dots,\ell - 1$, as required.
\end{proof}

\subsection{A lemma on $\alpha$-bounded functions}

The following technical lemma is required
for the proof of Theorem \ref{main1}.

\begin{lemma}\label{poly}
Let $\alpha,x_0 \geq 1$, and let $f$ be $\alpha$-bounded. Then the function
\[ g(x) = \frac{xf(x)}{x + f(x_0)}.\]
is $(\alpha + 1)$-bounded, $g(x) \leq x$ for $x \in [3,x_0]$, and $g(x) \leq f(x)$ for all $x \in [3,\infty)$.
\end{lemma}

\begin{proof}
By definition, $g(x) \leq f(x)$ for $x \in [3,\infty)$ and $g(x)\leq x$ for $x \in [3,x_0]$. Also,
since $f$ is non-decreasing and positive on $[3,\infty)$, $g$ is non-decreasing on $[3,\infty)$.
It remains to check that $g$ is $(\alpha + 1)$-bounded.
For $y \geq x \geq 3$, using that $y^{\alpha}f(x) \geq x^{\alpha}f(y)$, we find
$$y^{\alpha + 1} g(x) = \frac{y^{\alpha + 1}xf(x)}{x+f(x_0)}\geq \frac{x^{\alpha + 1} yf(y)}{x + f(x_0)} \geq x^{\alpha+1} g(y).
$$
Therefore $g$ is $(\alpha+1)$-bounded.
\end{proof}

\section{Proof of Theorem \ref{main1} }

It is enough to prove Theorem \ref{main1} for all $\varepsilon < 1/2$. Let  $\beta = \alpha + 1$, $\eta = \frac{\varepsilon}{2}$ and
$x_0 = \max\{2m,(\frac{12\beta}{\eta})^{\beta + 1}\}$.
Define $k_1 = k_1(\varepsilon,\alpha,m) = \frac{8}{\varepsilon} f(x_0)$.
Let $g$ be a $\beta$-bounded function in Lemma \ref{poly}. We prove the following claim:

\smallskip

{\bf Claim.} {\em For $k \geq 3$, every $k$-chromatic graph $G \in \mathcal{P}$ has a cycle of length at least $(1 - \eta)g(k)$.}

\smallskip

Once this claim is proved, Theorem \ref{main1}(i) follows since for $k \geq k_1$,
\[ (1 - \eta)g(k) = (1 - \tfrac{\varepsilon}{2}) \frac{kf(k)}{k +  f(x_0)} \geq  (1 - \tfrac{\varepsilon}{2}) \frac{kf(k)}{k + \varepsilon k/8 }    \geq (1 - \varepsilon)f(k),\]
as required. Also, if $k \geq k_1$, then by Lemma \ref{odd} every $k$-chromatic graph in $\mathcal{P}$ contains cycles of at least $(1 - \eta)g(\frac{k}{4}) \geq (1 - \varepsilon)f(\frac{k}{4})$
consecutive lengths, which gives Theorem \ref{main1}(ii). We prove the claim by induction on $k \geq 3$. For $k \leq x_0$, $g(k) \leq k$ from Lemma \ref{poly}, so in that case $G$ contains a $k$-critical
subgraph which has minimum degree at least $k-1$ and therefore also a cycle of length at least $k$.
This proves the claim for $k \leq x_0$. Now suppose $k > x_0$. Let $G^*$ be the graph obtained from $G$ in Lemma \ref{le1}.
 By Lemma \ref{le1}(c), it is sufficient to show that $G^*$ has a cycle of length at least $(1 -  \eta)g(k)$.
 Let $C$ be a longest cycle in $G^*-e^*$. By induction, $|C|\geq (1 - \eta)g(k-1)$. Let $G_1 = G^*[C]$ and $\chi_1 = \chi(G_1)$, and let $G_2=G^*-G_1 - e^*$ and  $\chi_2 := \chi(G_2)$.
 Take $C'$ to be a longest cycle in $G_2$. Let $S$ be a minimum vertex set covering all paths from $C$ to $C'$.
 Either $S$ separates $C'-S$ from $C-S$ or $S=V(C')$.
 Let $|S| = \ell$.
By Menger's Theorem, $G^*$ has $\ell $ vertex-disjoint paths $P_1,P_2,\dots,P_{\ell}$ between $C$ and $C'$
-- note $\ell \geq 3$, as $G^*$ is 3-connected. Let $H = \bigcup_{i = 1}^{\ell} P_i \cup C \cup C'$.
We find a cycle $C^* \subset H$ with
\begin{equation}\label{eq}
|C^*| \geq \frac{\ell - 1}{\ell}|C| + \frac{1}{2}|C'|.
\end{equation}
To see this, first note that two of the paths, say $P_i$ and $P_j$, contain ends at distance at most
$\frac{1}{\ell}|C|$ on $C$, and now $P_i \cup P_j \cup C \cup C'$
contains a cycle $C^*$ of length at least
$$\frac{\ell-1}{\ell}|C| + \frac{1}{2}|C'| + |P_i| + |P_j| \geq \frac{\ell - 1}{\ell}|C| + \frac{1}{2}|C'|.$$
At the same time, $H$ contains a cycle $C^{**}$ with
\begin{equation}\label{eq1}
|C^{**}| \geq \frac{2}{3}(|C| + |C'|),
\end{equation}
since there exist three cycles that together cover every edge of $P_1 \cup P_2 \cup P_3 \cup C \cup C'$ exactly twice, and one of them has the required length.
 Now we complete the proof in three cases.

\medskip

{\bf Case 1.} $\chi_1 \geq (1 - \frac{\eta}{\beta})k$. Then $\chi_1 \geq (1 - \eta)k \geq \frac{k}{2} \geq m \geq 3$, which implies $n_{\mathcal{P}}(\chi_1) \geq f(\chi_1)$.
Since $\eta \leq \frac{1}{2} < \beta$, we have $(1 - \frac{\eta}{\beta})^{\alpha} \geq 1 - \eta$. Since $f$ is $\alpha$-bounded,
\[ |C| \geq n_{\mathcal{P}}(\chi_1) \geq f(\chi_1) \geq (1 - \tfrac{\eta}{\beta})^{\alpha} f(k) \geq (1 - \eta)f(k) \geq (1 - \eta)g(k).\]

\medskip

{\bf Case 2.} $\chi_1  < (1 - \frac{\eta}{\beta})k$ and  $\chi_2 \geq (1 - \tfrac{1}{4\beta})k$. Since $\chi_2 \geq 3$ and $g$ is $\beta$-bounded,
\[ g(\chi_2) \geq (1 - \tfrac{1}{4\beta})^{\beta} g(k) \geq \tfrac{3}{4}g(k) \quad \mbox{ and } \quad g(k-1) \geq (\tfrac{k-1}{k})^{\beta}g(k) \geq (1-\tfrac{\beta}{k})g(k).\]
As $k > x_0 > 6\beta$ and $|C'| \geq (1 - \eta)g(\chi_2)$, we obtain from (\ref{eq1}):
\[ |C^{**}| \geq \tfrac{2}{3}(1 - \eta)g(k-1) + \tfrac{2}{3}(1 - \eta)g(\chi_2)
\geq (1 - \eta)g(k) \cdot (\tfrac{7}{6} - \tfrac{\beta}{k}) > (1 - \eta)g(k).\]

\medskip

{\bf Case 3.} $\chi_1  < (1 - \frac{\eta}{\beta})k$
 and  $\chi_2 < (1 - \tfrac{1}{4\beta})k$.
Then $\chi_2 \geq k-1-\chi_1 > \tfrac{\eta}{5\beta}k \geq 3$.
Since every $\chi_2$-chromatic graph contains a cycle of length at least $\chi_2$, we have that $|C'|\geq \chi_2$.
If $S=V(C')$, then $\ell=|C'|\geq \chi_2>\tfrac{\eta}{5\beta}k$.
Otherwise, by Lemma \ref{le0}, $\chi_2 \geq k - \min\{\chi_1,\ell\} - 1$,
and so $\ell > k - \chi_2 - 1 > \tfrac{\eta}{5\beta}k$. By (\ref{eq}), and since $g$ is $\beta$-bounded,
\begin{eqnarray*}
|C^*| &\geq& \tfrac{\ell - 1}{\ell}(1 - \eta)g(k-1) + \tfrac{1}{2}g(\chi_2) \\
&\geq& (1 - \eta)g(k) \cdot (\tfrac{\ell - 1}{\ell} (\tfrac{k-1}{k})^{\beta} + \tfrac{1}{2}(\tfrac{\eta}{4\beta})^{\beta})\\
&\geq& (1 - \eta)g(k) \cdot (1 - \tfrac{\beta}{k} -\tfrac{1}{\ell} + \tfrac{1}{2}(\tfrac{\eta}{6\beta})^{\beta}).
\end{eqnarray*}
Since $k \geq (\frac{12\beta}{\eta})^{\beta + 1}$ and $\ell > \tfrac{\eta}{5\beta}k$, $\frac{\beta}{k} + \frac{1}{\ell} \leq \frac{6\beta}{\eta k} < \frac{1}{2}(\frac{\eta}{6\beta})^{\beta}$, and $|C^*| > (1 - \eta)g(k)$, as required. \qed

\section{Proof of Theorem \ref{main}}\label{mainproof}

For the proof of Theorem \ref{main}, we use Theorem \ref{main1} with some specific choice of the function $f(k)$.
Let $\alpha(G)$ denote the independence number of a graph $G$. Shearer~\cite{Shearer} showed the following:

\begin{lemma}
For every $n$-vertex triangle-free graph $G$ with average degree $d$,
\begin{equation}\label{0319}
\alpha(G)> n\frac{\log (d/e)}{d}.
\end{equation}
\end{lemma}

This implies the following simple fact. In what follows, let $\phi(x) = (\frac{1}{2}x\log x)^{1/2}$.

\smallskip

\begin{lemma}\label{Sh}
If $G$ is an $n$-vertex triangle-free graph and $n \geq e^{2e^3}$, then $\alpha(G)\geq \phi(n)$.
\end{lemma}

\vspace{-0.2in}

\begin{proof}
Let $d = \phi(n)$. Then if $n \geq e^{2e^3}$,  $d > e(en)^{1/2}$.
Since $G$ is triangle-free, the neighborhood of
any vertex is an independent set, so we may assume $G$ has maximum degree less than $d$.
Then by~\eqref{0319} and since $d > e(en)^{1/2}$,
\[ d > n\frac{\log (d/e)}{d} > n \frac{\log (en)^{1/2}}{\phi(n)} = \phi(n) = d,\]
{a contradiction.}
\end{proof}

To find a lower bound on the number of vertices of a triangle-free $k$-chromatic graph, we require a lemma of
Jensen and Toft~\cite{J} (they took $s=2$, but their proof works for each positive
integer $s$):

\begin{lemma}[\cite{J}, Problem 7.3]\label{JT}
Let $s\geq 1$ and let $\psi\,: [s,\infty) \to (0,\infty)$ be a positive continuous nondecreasing function.
Let $\mathcal{P}$ be a monotone class of graphs such that $\alpha(G)\geq \psi(|V(G)|)$ for every
$G\in\mathcal{P}$ with $|V(G)|\geq s$.
Then for every such $G$ with $|V(G)| = n$,
\[ \chi(G) \leq s + \int_s^n \frac{1}{\psi(x)} dx.\]
\end{lemma}

\smallskip

\begin{lemma}\label{3bounded}
Let $f(x) = cx^2 \log x$ where $x \geq 3$ and $c > 0$. Then $f$ is 3-bounded.
\end{lemma}

\vspace{-0.2in}

\begin{proof}
The function $f$ is positive and non-decreasing, so one only has to check $y^3 f(x) \geq x^3 f(y)$
whenever $y \geq x \geq 3$. This follows from $y\log x \geq x\log y$ for $y \geq x \geq 3$, since the function
$\frac{y}{\log y}$ is increasing for $y \geq 3$.
\end{proof}

\begin{lemma}\label{ramsey}
For every $\delta > 0$, there exists $k_2(\delta)$ such that if $k \geq k_2(\delta)$, then every $k$-chromatic triangle-free graph has at least $(\tfrac{1}{4} - \delta)k^2 \log k$ vertices.
\end{lemma}

\vspace{-0.2in}

\begin{proof}
If $\delta \geq \frac{1}{4}$ the lemma is trivial, so suppose $\delta < \frac{1}{4}$.
Let $\gamma(x) = \frac{1}{2} - \frac{1}{2\log ex}$. Let $G$ be a $k$-chromatic triangle-free $n$-vertex graph.
We apply the preceding lemma with $\psi(x) = \phi(x)$ supplied by Lemma \ref{Sh}. For $s \geq e^{2e^3}$, and using $\gamma(s) \leq \gamma(x)$ for $x \geq s$:
\begin{eqnarray*}
\chi(G) &\leq& s + \sqrt{2}\int_{s}^n (x\log ex)^{-1/2} dx = s + \frac{\sqrt{2}}{\gamma(s)} \int_{s}^n  (x\log ex)^{-1/2}\gamma(s)dx \\
&\leq& s + \frac{\sqrt{2}}{\gamma(s)} \int_{s}^n (x\log x)^{-1/2} \gamma(x) dx.
\end{eqnarray*}
An antiderivative for the integrand is exactly $x^{1/2}(\log ex)^{-1/2}$, and therefore
\[ \chi(G) \leq s + \frac{\sqrt{2}}{\gamma(s)} n^{1/2}(\log en)^{-1/2}.\]
On the other hand, $\chi(G) \geq k$ so if $j = k - s$,
\[ n \geq \gamma(s)^2 j^2 \log(\gamma(s)j).\]
If $s = \lceil \max\{e^{2e^3},e^{1/\delta}\}\rceil$, then $\gamma(s) \geq \frac{1}{2}(1 - \delta)$, so since $\delta < \frac{1}{4}$,
\[ n \geq \tfrac{1}{4}(1 - \delta)^2 j^2 \log \tfrac{1}{4}j \geq \tfrac{1}{4}(1 - \delta)^2 j^2 \log j - k^2.\]
If $j \geq 3$, then by Lemma \ref{3bounded},
\[ n \geq \tfrac{1}{4}(1 - \delta)^2(\tfrac{j}{k})^3 k^2 \log k - k^2 \geq \tfrac{1}{4}(1 - 2\delta - \tfrac{3s}{k})k^2 \log k - k^2.\]
Let $k_2(\delta) = s^4 \geq \max\{e^{8e^3},e^{4/\delta}\}$. Since $k \geq e^{4/\delta}$, $k^2 \leq \frac{1}{4}\delta k^2 \log k$ and $3s \leq \delta k$. Therefore
\[ n \geq \tfrac{1}{4}(1 - 2\delta - \tfrac{3s}{k})k^2 \log k - k^2 \geq \tfrac{1}{4}(1 - 4\delta)k^2 \log k.\]
This completes the proof.
\end{proof}

\bigskip

{\bf Proof of Theorem \ref{main}.} The theorem is trivial if $\varepsilon \geq \frac{1}{4}$, so we assume $\varepsilon < \frac{1}{4}$. We will derive
Theorem \ref{main} from Theorem \ref{main1}. Let $\eta = 2\varepsilon,\delta =\frac{\varepsilon}{2}$.
By Lemma~\ref{ramsey}, for $k \geq m := k_2(\delta)$, every triangle-free $k$-chromatic graph $G$ has
at least $(\frac{1}{4} - \delta)k^2 \log k$ vertices. Then $f(x) = (\frac{1}{4} - \delta)x^2 \log x$ is $3$-bounded, by Lemma \ref{3bounded}.
By Theorem~\ref{main1}, with $\alpha = 3$, $G$ contains a cycle of length at least $(1 - \eta)f(k)$ as well
as cycles of at least $(1 - \eta)f(\frac{k}{4})$ consecutive lengths in $G$, provided $k \geq k_1(\eta,\alpha,m)$.
Letting $k_0(\varepsilon) = k_1(\eta,\alpha,m) = k_1(2\varepsilon,3,k_2(\frac{\varepsilon}{2}))$, and noting $(1 - \eta)f(k) \geq (\frac{1}{4} - \varepsilon)k^2 \log k$ by the choice
of $\eta$ and $\delta$, we have a cycle of length at least $(\frac{1}{4} - \varepsilon)k^2 \log k$ in $G$ whenever $k \geq k_0(\varepsilon)$. Similarly, if $k$
is large enough relative to $\varepsilon$, then $G$ contains cycles of at least $(1 - \eta)f(\frac{k}{4}) \geq (\frac{1}{64} - \varepsilon)k^2 \log k$ consecutive lengths. This completes the proof. \qed

\section{Proof of Theorem \ref{main2}}\label{main2proof}

\begin{lemma}\label{kr}
Let $r\geq 3$ and $G$ be a $K_r$-free $n$-vertex graph. Then $\alpha(G) \geq n^{1/(r - 1)}-1$.
\end{lemma}
\begin{proof} If $r\geq 3$ and $n\leq 2^{r-1}$, then ${n^{1/(r - 1)}}-1\leq 1$, so the claim holds. Let $n> 2^{r-1} $.
If $r = 3$, either  $\Delta(G)\geq n^{1/2}$ or the graph is greedily $\lfloor n^{1/2}\rfloor+1$-colorable.
Since vertex neighborhoods are independent sets in a triangle-free graph, either case gives an independent set of size at least $n^{1/2}-1$.
For $r > 3$, either the graph has a vertex $v$ of degree at least $d \geq n^{(r - 2)/(r - 1)}$, or the graph is $n^{1 - 1/(r - 1)}+1$-colorable.
In the latter case, the largest color class is an independent set of size at least $n^{1/(r - 1)}-1$. In the former case, since
the neighborhood of $v$ induces a $K_{r - 1}$-free graph, by induction it contains an independent set of size at least $d^{1/(r - 2)} -1\geq n^{1/(r - 1)}-1$,
as required.
\end{proof}

\begin{lemma}\label{color}
Let $r\geq 3$ and $G$ be a $K_r$-free $n$-vertex graph. Then
\[ \chi(G) < 4n^{1 - 1/(r-1)}.\]
\end{lemma}

\begin{proof} The function $f(x)=\max\{1,{x^{1/(r - 1)}}-1\}$ is positive continuous and nondecreasing. Since each nontrivial graph has
an independent set of size $1$, by Lemmas~\ref{JT} and~\ref{kr},
$$\chi(G) \leq 1 + \int_1^{n}\frac{1}{f(x)} {\rm d}x\leq 1+ \int_1^{n}\frac{2}{x^{1/(r - 1)}} {\rm d}x < 4n^{1 - 1/(r-1)}.$$

\vspace{-0.48in}
\end{proof}

\medskip

{\bf Proof of Theorem \ref{main2}.} We prove the first claim of the theorem for all $k,r\geq 3$, and
 then apply Lemma \ref{odd} to prove the second claim.
Let $G$ be an $n$-vertex $K_{r + 1}$-free graph with $\chi(G) = k \geq 3$.   By Lemma \ref{color}, $k < 4n^{1 - \frac{1}{r}}$, so $|V(G)| \geq (\tfrac{k}{4})^{\frac{r}{r - 1}}:= f(k)$.
Since $f$ is $\frac{r}{r - 1}$-bounded, the proof is complete
by Theorem \ref{main1} with $\mathcal{P}$ the property of $K_{r + 1}$-free graphs. \qed

\section{Concluding remarks}

$\bullet$ In this paper, we have shown that the length of a longest cycle and the length of a longest interval of lengths of cycles in $k$-chromatic
graphs $G$ are large when $G$ lacks certain subgraphs. In particular, when $G$ has no triangles,
this yields a proof of Conjecture~1 (in a stronger form). We believe that the following holds.

\begin{conj}
Let $G$ be a $k$-chromatic triangle-free graph and let $n_k$ be the minimum number of vertices in a $k$-chromatic
triangle-free graph. Then $G$ contains a cycle of length at least $n_k - o(n_k)$.
\end{conj}

$\bullet$ If Shearer's bound~\cite{Shearer} is tight, i.e.,
$n_k \sim \frac{1}{4}k^2 \log k$, then the lower bound on the length of the longest cycle in any $k$-chromatic triangle-free
graph in Theorem \ref{main} would be tight.

\end{document}